\documentclass[11pt,a4paper]{amsart}

\usepackage[english]{babel}
\usepackage[utf8]{inputenc}
\usepackage[T1]{fontenc}
\usepackage{amsmath}
\usepackage{amssymb}
\usepackage{multirow}
\usepackage{alltt}
\usepackage{graphicx}
\usepackage{bbm}
\usepackage{centernot}
\usepackage{amsthm}
\usepackage{mathrsfs}
\usepackage{textcomp}
\usepackage{calc}
\usepackage{fullpage}
\usepackage{amsrefs}
\usepackage{enumitem}

\usepackage[usenames,dvipsnames]{xcolor}
\definecolor{darkblue}{RGB}{0,0,170}
\definecolor{brickred}{RGB}{200,0,0}
\usepackage[hyperindex, linktocpage]{hyperref}
\hypersetup{colorlinks=true, linkcolor=darkblue, citecolor=brickred}

\usepackage{tikz}

\newcommand{\R}{\mathbb{R}}
\newcommand{\N}{\mathbb{N}}

\newcommand{\Z}{\mathbb{Z}}
\newcommand{\eps}{\varepsilon}
\newcommand{\dist}{\mathrm{dist}}

\newcommand{\Ds}{{\left(-\lapl\right)}^s}
\newcommand{\lapl}{\Delta}

\newcommand{\1}{\mathbf{1}}

\newcommand{\cC}{{\mathcal C}}

\newcommand{\cE}{{\mathcal E}}
\newcommand{\cF}{{\mathcal F}}

\newcommand{\cH}{{\mathcal H}}

\newcommand{\cM}{{\mathcal M}}

\newcommand{\intgr}[1]{\lfloor #1\rfloor}

\newtheorem{theorem}{Theorem}[section]
\newtheorem{lemma}[theorem]{Lemma}
\newtheorem{proposition}[theorem]{Proposition}

\theoremstyle{remark}
\newtheorem{remark}[theorem]{Remark}

\theoremstyle{definition}

\newtheorem{definition}[theorem]{Definition}

\allowdisplaybreaks

\date{\today}
\author{Nicola Abatangelo}
\address{(N. Abatangelo) Dipartimento di Matematica, Alma Mater Studiorum Università di Bologna, P.zza di Porta S. Donato~5, 40126 Bologna, Italy.}
\email{nicola.abatangelo@unibo.it}
\author{Sven Jarohs}
\address{(S. Jarohs) Institut für Mathematik, Goethe-Universität Frankfurt am Main, Robert-Meyer-Str.~10, 60325 Frankfurt am Main, Germany.}
\email{jarohs@math.uni-frankfurt.de}

\thanks{\textit{MSC2020}:
\textit{Primary:}
35A23, 
35G15, 
35P05; 
\textit{Secondary:}
47A75, 
49R20, 
49R50. 
}

\title{Oscillatory phenomena for higher-order fractional Laplacians}

\begin{document}

\begin{abstract}
We collect some peculiarities of higher-order fractional Laplacians
$\Ds,\ s>1,$ with special attention to the range~$s\in(1,2)$, 
which show their oscillatory nature. These include
the failure of the polarization and Pólya-Szeg\H o inequalities and
the explicit example of a domain with sign-changing first eigenfunction.
In spite of these fluctuating behaviours, we prove how the Faber-Krahn inequality
still holds for any~$s>1$ in dimension one.
\end{abstract}

\maketitle

\section{Introduction}


The fractional Laplacian operator, usually denoted by~$\Ds,\ s\in(0,1),$
is the nonlocal integral operator
\begin{align}\label{fraclapl01}
\Ds u(x)=a_{n,s}\lim_{\eps\downarrow 0}\int_{\{y\in\R^n:|y|>\eps\}}\frac{u(x)-u(x+y)}{{|y|}^{n+2s}}\;dy.
\end{align}
It is naturally associated to the stochastic analysis of some Lévy processes, 
in particular of the~$\alpha$-stable ones, since it arises as their infinitesimal generator.
For this reason, it has been used in models from mathematical finance, population dynamics, quantum mechanics, and electrostatics, among others. We refer to~\cites{av,bv,hitchhiker,garofalo} for an introduction to the basic features and a more comprehensive list of applications of the operator.
From an abstract perspective, the family~$\{\Ds\}_{s\in(0,1)}$ can be thought of as a collection of operators interpolating the identity and the Laplacian, indeed, at least for~$u\in C^\infty_c(\R^n)$,
\begin{align*}
\lim_{s\downarrow 0}\Ds u & = u, \\
\lim_{s\uparrow 1} \Ds u & = -\lapl u,
\end{align*}
see~\cite{hitchhiker}*{Proposition 4.4}. 

Here, we are mainly concerned with powers of the Laplace operator greater than~$1$, which we call \textit{higher-order fractional Laplacians}. Although the integral expression in~\eqref{fraclapl01} cannot be extended to~$s>1$ in a straightforward way, it is possible to define~$\Ds,\ s>1,$ as the pseudo-differential operator
\begin{align}\label{fraclapl-pseudo}
\cF\big[\Ds u\big](\xi) = {|\xi|}^{2s}\cF u(\xi)
\qquad \text{for }\xi\in\R^n,\ u\in C^\infty_c(\R^n),
\end{align}
where~$\cF$ denotes the Fourier transform.
For~$s\in(0,1)$, the above and~\eqref{fraclapl01} coincide by properly choosing the normalizing constant~$c_{n,s}$, see~\cite{samko}*{Chapter 5, Lemma 25.3 and Theorem 26.1}. 
With definition~\eqref{fraclapl-pseudo}, the family~$\{\Ds\}_{s>0}$ interpolates the (local) \textit{polyharmonic operators}---or also \textit{polylaplacians}---~$\Ds,\ s\in\N,$ for which we refer to the monograph~\cite{ggs}.

Boundary problems driven by a polylaplacian require several boundary conditions in order to be well-posed, and this is due to the high order of the operator. One way of prescribing boundary conditions is imposing that the solution~$u$ satisfies
\begin{align*}
u=\frac{\partial u}{\partial\nu}=\ldots=\frac{\partial^{s-1}u}{\partial\nu^{s-1}}=0
\qquad \text{at the boundary, for }s\in\N,
\end{align*}
where~$\nu$ denotes the normal unit vector to the boundary. These go under the name \textit{Dirichlet conditions} and they represent the pointwise analogue of variationally solving equations in the Sobolev space~$H^s_0$.

When coupled with Dirichlet conditions, polylaplacians present an oscillatory behaviour. This may be exemplified by the failure of the weak maximum principle
\begin{align}\label{maxprinc}
\Ds u\geq 0 \quad\text{in }\Omega,\quad u\in H^s_0(\Omega),\ s\in\N\setminus\{1\},
\qquad\centernot\Longrightarrow\qquad u\geq 0\quad\text{in }\Omega.
\end{align}
It was first remarked in 1908~\cite{hada1908} that, in dimension~$n=2$, there are annular domains for which~\eqref{maxprinc} occurs for the bilaplacian~$s=2$.
It was then conjectured that~\eqref{maxprinc} was \textit{not} the case for convex domains: this is known in the literature as the \textit{Boggio-Hadamard conjecture}. But this conjecture was proved to be false~\cite{duffin49} for some rectangular domains and, later on, several other counterexamples were built; these include for~$s,n=2$: an infinite strip~\cite{MR144069}, (most of) infinite wedges~\cite{MR325989}, the punctured disk~\cite{MR481048}, and eccentric ellipses~\cite{garabedian}. As to this one last counterexample, it is worth mentioning that, on top of the simple geometry of the domain,  it is possible to provide completely elementary counterexamples in terms of polynomials~\cite{MR1267051}: this highlights how the above mentioned oscillations are deeply written in the nature of higher-order operators and they do not arise as singular phenomena. For~$s=3$ we refer to~\cite{sweers-triharm} and for~$s=4$ to~\cite{sweers-correction} for elementary explicit counterexamples. For~$s\in\N$ even and any space dimension, the lack of maximum principles can be deduced by the analysis of the oscillations of the first eigenfunction carried out in~\cite{kkm}.

Another interesting (and extremely difficult) problem for the bilaplacian with Dirichlet conditions is the \textit{Faber-Krahn inequality}. It states the first eigenvalue of any domain is larger than the first eigenvalue of the ball with the same measure. This has been proved independently for~$n=2$ in~\cite{nadirashvili} and for~$n=2,3$ in~\cite{ashbaugh-benguria}, while it remains completely open for higher dimensions. Both~\cite{nadirashvili} and~\cite{ashbaugh-benguria} are based on the previous analysis carried out in~\cite{talenti1} which, in turn, borrows the idea of \textit{Talenti's principle}~\cite{talenti-principle}*{Theorem~1}.

As the collection~$\{\Ds\}_{s\in(1,2)}$ connects the Laplacian with the bilaplacian, one could legitimately wonder in which way this oscillatory behaviour appears in the transition from~$s=1$ to~$s=2$. Our analysis here aims at contributing to answer this question, by showing that as soon as~$s>1$ oscillatory phenomena emerge at a large extent. 

First of all, let us recall that, on the fractional and nonlocal side~$s\in(1,\infty)\setminus\N$, it was first proved in~\cite{ajs-maxprinc} that ``almost'' no disconnected domain can satisfy the weak maximum principle whenever~$s\in(2k-1,2k)$ with~$k\in\N$: the prototypical domain here can be thought as the union of two disjoint balls (mind that the disconnection of the domain is overruled by the nonlocality of the operator).
A somewhat more precise analysis in~\cite{ajs-repr}*{Theorem 1.10} studied the sign of the Green function of such two-balls domain, noticing in particular how the Green function is positive (if the balls have equal radii) for~$s\in(2k,2k+1),\ k\in\N$, and thus the above mentioned counterexample cannot extend also to this range of~$s$.
Later,~\cite{ajs-ellipse} constructed counterexamples on eccentric ellipsoids for~$s\in(1,\sqrt{3}+3/2)$, therefore covering the exponents~$s\in(2,3)$.

Another interesting property found in~\cite{musina-nazarov} deals with the Dirichlet energy\footnote{This could be equivalently introduced as the quadratic form which is naturally associated to~$\Ds$ as defined in~\eqref{fraclapl-pseudo}. For~$s=0$ it coincides with the square of the~$L^2$ norm.}
\begin{align}\label{energy-fourier}
\cE_s(u,u)=\int_{\R^n}|\xi|^{2s}\big|\cF u(\xi)\big|^2\;d\xi
\end{align}
of the Nemytskii operator~$u\mapsto|u|$: it holds
\begin{align*}
\cE_s(u,u)\leq\cE_s(|u|,|u|)
\qquad\text{for }s\in(1,3/2)
\end{align*}
where the equality holds only if~$u$ is of constant sign. This is in particular equivalent to
\begin{align*}
\cE_s(u_+,u_-)\geq 0
\qquad\text{for }s\in(1,3/2)
\end{align*}
where~$u_+,u_-$ denote respectively the positive and the negative parts of~$u$.

The purpose of the present note is to collect a series of remarks on the behaviour of higher-order fractional Laplacians. In particular, we show that:
\begin{itemize}
\item The polarization inequality for~$s\in(1,3/2)$ is reversed with respect to what happens for~$s\in(0,1]$---see Definition~\ref{def:polarization} and Theorem~\ref{polarization-inequality}; the polarization of a function is a rearrangement with respect to a hyperplane and it is a useful tool in proving symmetry of solutions to elliptic and parabolic problems: we refer to~\cite{W10} and the many references therein for a survey about this matter; among the applications of the polarization, it is worth mentioning that it can approximate several other notions of rearrangement (including the spherical decreasing one) and that, as a consequence, a number of rearrangement inequalities can be deduced from the polarization inequality, see~\cite{vS};
\item The Pólya-Szeg\H o inequality fails for~$s\in(1,\infty)$, see Paragraph~\ref{par:polyaszego-counter}: contrarily to what happens for the polarization, the Pólya-Szeg\H o inequality is not reversed, see Paragraph~\ref{par:polyaszego-ex}; this inequality compares the energy of the spherical decreasing rearrangement (or, also, Schwarz symmetrization)---see Definition~\ref{def:schwarz}--- of a function with the energy of the original function; for~$s\in(0,1)$ the Pólya-Szeg\H o inequality has been proved in~\cite{park}*{equation (14)}; recall that this inequality implies the Faber-Krahn inequality;
\item On a two-balls domain, the first eigenvalue is simple and the first eigenfunction is symmetric and positive for~$s\in(2k,2k+1),\ k\in\N$, whilst it is anti-symmetric and positive on one ball (and negative on the other) for~$s\in(2k-1,2k),\ k\in\N$, see Theorem~\ref{thm:first-twoball};
\item Again on a two-balls domain, when the balls have equal radii, maximum principles are recovered for anti-symmetric ``positive'' data for~$s\in(2k,2k+1),\ k\in\N$, and for symmetric positive data for~$s\in(2k-1,2k),\ k\in\N$, see Proposition~\ref{part-max-princ-prelim};
\item Despite the results concerning the polarization and the Pólya-Szeg\H o inequalities, the Faber-Krahn inequality still holds for any~$s\in(1,\infty)$ at least in dimension~$n=1$, see Theorems~\ref{thm:faber-krahn even} and~\ref{thm:faber-krahn odd}.
\end{itemize}

\section{Notations and recallings about higher-order fractional Laplacians}

In the following we will need some notations and known facts about the analysis of higher-order fractional Laplacians, which we recall here below for future reference.

We will denote by~$\intgr{\cdot}$ the integer part of a real number, \textit{i.e.},
\begin{align*}
\intgr{s}=\max\{m\in\Z:m<s\},\qquad s\in\R:
\end{align*}
remark that, with this definition,~$\intgr{s}=s-1$ whenever~$s\in\N$. If~$\N$ denotes the set of positive integer numbers, then~$\N_0=\N\cup\{0\}$ and~$2\N_0$,~$2\N_0+1$ denote respectively the sets of even and odd non-negative integers. Symbols~$\vee$ and~$\wedge$ will denote respectively the max and min operations between real numbers, namely
\begin{align*}
a\vee b=\max\{a,b\},\qquad a\wedge b=\min\{a,b\},\qquad\text{for }a,b\in\R.
\end{align*}
By~$\omega_n$ we mean the volume of the unit ball in~$\R^n$,
\begin{align*}
\omega_n=\frac{2\pi^{n/2}}{n\,\Gamma(n/2)}.
\end{align*}
Given a measurable~$A\subset\R^n$, we denote by~$\1_A$ its characteristic function.

The fractional Laplacian~$\Ds$ as defined in~\eqref{fraclapl-pseudo} admits the pointwise representation, see~\cite{ajs-hypersingular},
\begin{align*}
\Ds u(x):=\frac{\kappa_{n,s}}{2}\int_{\R^n} \frac{\delta_{\intgr{s}+1} u(x,y)}{|y|^{n+2s}} \ dy,
\end{align*}
where
\begin{align*}
\delta_{\intgr{s}+1} u(x,y):= \sum_{k=-\intgr{s}-1}^{\intgr{s}+1} (-1)^k { \binom{2\intgr{s}+2}{\intgr{s}+1-k}} u(x+ky) \qquad \text{for }x,y\in \R^n
\end{align*}
is a finite difference of order~$2\intgr{s}+2$, and~$c_{n,s}$ is the positive constant given by
\begin{align*}
\kappa_{n,s}
:=\left\{\begin{aligned}
&\frac{2^{2s}\Gamma(n/2+s)}{ \pi^{n/2}\Gamma(-s) }\Bigg(\sum_{k=1}^{\intgr{s}+1}(-1)^{k}{ \binom{2\intgr{s}+2}{\intgr{s}+1-k}} k^{2s}\Bigg)^{-1}, && s\not\in\N,\\
&\frac{2^{2s}\Gamma(n/2+s)\,s!}{2\pi^{n/2}} \Bigg(\sum_{k=2}^{s+1} (-1)^{k-s+1}{\binom{2s+2}{s+1-k}} k^{2s} \ln(k)\Bigg)^{-1},&& s\in\N.
\end{aligned}\right.
\end{align*}

Recalling the quadratic form in~\eqref{energy-fourier}, the fractional Sobolev space~$H^s(\R^n)$ is defined as
\begin{align*}
H^s(\R^n)=\big\{u\in L^2(\R^n):\cE_s(u,u)<\infty\big\}.
\end{align*}
On a bounded open domain~$\Omega\subset\R^n$ the Sobolev space associated with homogeneous boundary conditions is
\begin{align*}
\cH^s_0(\Omega)=\big\{u\in H^s(\R^n):u=0\text{ in }\R^n\setminus\overline\Omega\big\}.
\end{align*}
The bilinear form induced by~\eqref{energy-fourier} is
\begin{align*}
\cE_s(u,v)=\int_{\R^n}{|\xi|}^{2s}\,\cF u(\xi)\,\overline{\cF v(\xi)}\;d\xi
\end{align*}
and it admits also the following equivalent representations
\begin{align}\label{eq:energies}
\cE_s(u,v)=\left\lbrace\begin{aligned}
& \cE_{s-\intgr{s}}\Big({(-\lapl)}^{\intgr{s}/2}u,{(-\lapl)}^{\intgr{s}/2}v\Big) && \text{if }\intgr{s}\in2\N_0, \\ 
& \cE_{s-\intgr{s}}\Big(\nabla{(-\lapl)}^{(\intgr{s}-1)/2}u,\nabla{(-\lapl)}^{(\intgr{s}-1)/2}v\Big) && \text{if }\intgr{s}\in2\N_0+1.
\end{aligned}\right.
\end{align}
A consequence of this representation is that, if~$u,v\in H^s(\R^n)$ with~$uv\equiv 0$ in~$\R^n$, then
\begin{align}\label{disjointsupport}
\cE_s(u,v)={(-1)}^{\intgr{s}+1}\,\frac{c_{n,s}}{2}\int_{\R^n}\int_{\R^n}\frac{u(x)\,v(y)}{{|x-y|}^{n+2s}}\;dx\;dy,
\end{align}
where
\begin{align*}
c_{n,s}=\frac{2^{2s}\Gamma(n/2+s)}{\pi^{n/2}\big|\Gamma(-s)\big|}>0,
\end{align*}
see~\cite{ajs-maxprinc}*{Lemma 4.4} and~\cite{musina-nazarov}*{Theorem 2}.
Equation~\eqref{disjointsupport} should be borne in mind, as it will play a crucial role in the following.

\section{A (reversed) polarization inequality}

In the following, let $\Sigma$ be an open half-space in $\R^n$. Denote by~$\tau_\Sigma:\R^n\to\R^n$ the reflection around~$\partial \Sigma$. 
\begin{definition}\label{def:polarization}
The \textit{polarization} of a function~$u:\R^n\to\R$ is given by
\begin{align*}
u_\Sigma:\R^n\to\R,\quad u_\Sigma(x)=
\left\{\begin{aligned} 
& u(x)\vee u\big(\tau_\Sigma(x)\big) && \text{for } x\in \Sigma,\\
& u(x)\wedge u\big(\tau_\Sigma(x)\big) && \text{for } x\in \R^n\setminus \Sigma.
\end{aligned}\right.
\end{align*}
\end{definition}
Note that we may also write
\begin{equation}\label{polarization-different}
u_\Sigma(x)=\frac{u(x)+u\big(\tau_\Sigma(x)\big)}{2}
\ \left\{\begin{aligned} 
& +\frac{\big|u(x)-u\big(\tau_\Sigma(x)\big)\big|}{2} && \text{for } x\in \Sigma \\
& -\frac{\big|u(x)-u\big(\tau_\Sigma(x)\big)\big|}{2} && \text{for } x\in \R^n\setminus \Sigma,
\end{aligned}\right.
\end{equation}
and it holds
\begin{align}\label{eq:anti-remark}
u_\Sigma+u_\Sigma\circ\tau_\Sigma=u+u\circ\tau_\Sigma \quad\text{in }\R^n.
\end{align}
Recall that~$\cE_s(u,u)=\cE_s(u_\Sigma,u_\Sigma)$ for~$s=0,1$ (see, \textit{e.g.},~\cite{W10}) 
whereas, for~$s\in(0,1)$, we have
\begin{equation}\label{polarization01}
\cE_s(u_\Sigma,u_\Sigma)\leq \cE_s(u,u)
\end{equation}
where the inequality is strict if~$u_\Sigma\neq u$ (see, \textit{e.g.},~\cites{B94,vSW04,GJ19}). 
In the following, we consider~$s\in(1,3/2)$ and show that~\eqref{polarization01} is reversed, see Theorem~\ref{polarization-inequality} below.
The fact that~$u_\Sigma\in H^s(\R^n)$ whenever~$u\in H^s(\R^n)$ follows immediately from~\eqref{polarization-different} and~\cite{BM91}*{Théorème 1}.

\begin{lemma}
Let~$s\in(0,1)$,~$\Sigma$ an open half-space in $\R^n$, and~$v\in H^s(\R^n)$ such that~$v\circ \tau_\Sigma=-v$, \emph{i.e.},~$v$ is antisymmetric with respect to~$\partial \Sigma$. Then~$v\1_\Sigma\in \cH^s_0(\Sigma)$ and
\[
\cE_s(v\1_\Sigma,v\1_\Sigma)\leq \cE_s(v,v).
\]
\end{lemma}
\begin{proof}
The statement follows from the proof of~\cite{JW16}*{Lemma 3.2} (see also~\cite{J16}*{Lemma 3.2}).
\end{proof}

\begin{proposition}\label{lemma:antisymmetric}
Let~$s\in(1,3/2)$,~$\Sigma$ an open half-space in $\R^n$, and~$v\in H^s(\R^n)$ such that~$v\circ \tau_\Sigma=-v$, \emph{i.e.},~$v$ is antisymmetric w.r.t.~$\partial \Sigma$. Then~$v\1_\Sigma\in \cH^s_0(\Sigma)$.
\end{proposition}
\begin{proof}
By translation and rotation invariance, we may assume that~$\Sigma=\{x_1>0\}$, so that $\tau_\Sigma(x)=\tau_\Sigma(x_1,x')=(-x_1,x')$. Note that~$v\1_\Sigma\in \cH^1_0(\Sigma)$. Next recall that by~\cite{G11}*{Corollary 1.4.4.5} we have~$\cH^{s-1}_0(\Sigma)=H^{s-1}(\Sigma)$ for~$s\in(1,3/2)$. In particular, for any~$i=1,\ldots,n$, we have~$(\partial_iv)\1_\Sigma\in H^{s-1}(\Sigma)$ and thus also~$\partial_i(v\1_\Sigma)=(\partial_iv)\1_\Sigma\in \cH^{s-1}_0(\Sigma)$.
\end{proof}

\begin{remark}
We note that the conclusion of Proposition~\ref{lemma:antisymmetric} does not hold for~$s\geq 3/2$. Indeed, for example in the case~$n=1$ and~$s=2$, if~$\Sigma=\{x_1>0\}$ and~$v\in H^2(\R)$ is odd, then in general~$v'(0)\neq 0$ and thus~$v\1_\Sigma$ cannot be an element of~$\cH^2_0(\Sigma)$. More generally, one can check~\cite{MR2376460}*{Example 3.1.(i)} or~\cite{musina-nazarov}*{Example 1}.
\end{remark}

\begin{lemma}\label{lemma:disjoint:support}
	Let~$s\in(1,2)$,~$\Sigma$ an open half-space in $\R^n$, and~$w_1,w_2\in \cH^s_0(\Sigma)$ with~$w_1,w_2\geq 0$ and~$w_1w_2\equiv 0$ in~$\R^n$. Then
	\[
	\cE_s(w_1,w_2\circ \tau_\Sigma)\leq \cE_s(w_1,w_2)
	\]
	and the inequality is strict if~$w_1$ and~$w_2$ are nontrivial.
\end{lemma}
\begin{proof}
	By density, we may think of~$w_1,w_2\in C^{\infty}_c(\Sigma)$ with disjoint supports. Moreover, since $ w_1w_2\equiv 0$ in~$\R^n$ and~$w_i\1_{\R^n\setminus \Sigma}\equiv 0$ in~$\R^n$,~$i=1,2$, we also have~$w_1(w_2\circ \tau_\Sigma)\equiv 0$ in~$\R^n$. Hence,~\eqref{disjointsupport} implies
	\begin{align*}
	\cE_s(w_1,w_2\circ \tau_\Sigma) &= 
	\frac{c_{n,s}}2\int_{\Sigma}\int_{\R^n\setminus \Sigma}\frac{w_1(x) \, w_2\big(\tau_\Sigma(y)\big)}{{|x-y|}^{n+2s}}\;dy\;dx
	=\frac{c_{n,s}}2\int_{\Sigma}\int_{\Sigma}\frac{w_1(x) \, w_2(y)}{{|x-\tau_\Sigma(y)|}^{n+2s}}\;dy\;dx \\
	& \leq \frac{c_{n,s}}2\int_{\Sigma}\int_{\Sigma}\frac{w_1(x) \, w_2(y)}{{|x-y|}^{n+2s}}\;dy\;dx=\cE_s(w_1,w_2),
	\end{align*}
	since~$|x-\tau_\Sigma(y)|\geq |x-y|$ for~$x,y\in \Sigma$.
\end{proof}

\begin{lemma}\label{lemma:first separation}
Let~$s\in(1,3/2)$,~$\Sigma$ an open half-space in $\R^n$, and~$w\in H^s(\R^n)$ such that~$w\circ \tau_\Sigma=-w$, \emph{i.e.},~$w$ is antisymmetric with respect to~$\partial \Sigma$. Then
\begin{equation}\label{eq:anti-sym}
\cE_s(w,w)\leq \cE_s(|w|\1_\Sigma-|w|\1_{\R^n\setminus \Sigma},|w|\1_H-|w|\1_{\R^n\setminus \Sigma}).
\end{equation}
Moreover, the inequality is strict if~$w\1_\Sigma$ is sign-changing.
\end{lemma}
\begin{proof}
By Proposition~\ref{lemma:antisymmetric} we have~$w\1_\Sigma\in \cH^s_0(\Sigma)$ and hence also~$(w\1_\Sigma)^{\pm}=w^{\pm}\1_\Sigma\in \cH^s_0(\Sigma)$ by~\cite{BM91}*{Théorème 1}, since~$s<3/2$. Splitting~$w=w^{+}\1_\Sigma-w^-\1_\Sigma+w^+\1_{\R^n\setminus \Sigma}-w^-\1_{\R^n\setminus \Sigma}$ and~$|w|=w^++w^-$ we have that~\eqref{eq:anti-sym} is equivalent to
\begin{align*}
&-\cE_s(w^+\1_\Sigma,w^-\1_\Sigma)+\cE_s(w^+\1_\Sigma,w^+\1_{\R^n\setminus \Sigma})-\cE_s(w^+\1_\Sigma,w^-\1_{\R^n\setminus \Sigma}) \\
&-\cE_s(w^-\1_\Sigma,w^+\1_{\R^n\setminus \Sigma})+\cE_s(w^-\1_\Sigma,w^-\1_{\R^n\setminus \Sigma})-\cE_s(w^+\1_{\R^n\setminus \Sigma},w^-\1_{\R^n\setminus \Sigma})\\
&\qquad\leq \cE_s(w^+\1_\Sigma,w^-\1_\Sigma)-\cE_s(w^+\1_\Sigma,w^+\1_{\R^n\setminus \Sigma})-\cE_s(w^+\1_\Sigma,w^-\1_{\R^n\setminus \Sigma})\\
&\qquad\quad -\cE_s(w^-\1_\Sigma,w^+\1_{\R^n\setminus \Sigma})-\cE_s(w^-\1_\Sigma,w^-\1_{\R^n\setminus \Sigma})+\cE_s(w^+\1_{\R^n\setminus \Sigma},w^-\1_{\R^n\setminus \Sigma}),
\end{align*}
or, simply,
\[
\cE_s(w^+\1_\Sigma,w^+\1_{\R^n\setminus \Sigma})+\cE_s(w^-\1_\Sigma,w^-\1_{\R^n\setminus \Sigma})\leq 
\cE_s(w^+\1_\Sigma,w^+\1_{\Sigma})+\cE_s(w^+\1_{\R^n\setminus \Sigma},w^-\1_{\R^n\setminus \Sigma}).
\]
Using the anti-symmetry of~$w$, one has that
\[
w^+\1_{\R^n\setminus \Sigma}=(w^-\1_\Sigma)\circ \tau_\Sigma
\quad\text{and}\quad 
w^-\1_{\R^n\setminus \Sigma}=(w^+\1_\Sigma)\circ \tau_\Sigma,
\]
so that we only need to show
\[
\cE_s(w^+\1_\Sigma,(w^-\1_{\Sigma})\circ\tau_\Sigma)\leq \cE_s(w^+\1_\Sigma,w^-\1_{\Sigma}),
\]
since the scalar product is invariant under reflections. This last inequality holds in view of Lemma~\ref{lemma:disjoint:support}.
\end{proof}

\begin{proposition}\label{polarization1}
Let~$s\in(1,3/2)$,~$\Sigma$ an open half-space in $\R^n$, and~$u\in H^s(\R^n)$. Then
\[
\cE_s(u_\Sigma,u_\Sigma\circ\tau_\Sigma)\leq \cE_s(u,u\circ \tau_\Sigma).
\]
Moreover, this inequality is strict if~$u_\Sigma\neq u$.
\end{proposition}
\begin{proof}
First note that
\[
4\cE_s(u,u\circ\tau_\Sigma)=\cE_s(u+u\circ\tau_\Sigma,u+u\circ\tau_\Sigma)-\cE_s(u-u\circ\tau_\Sigma,u-u\circ\tau_\Sigma).
\]
Using this observation and~\eqref{polarization-different}, it follows that
\[
4\cE_s(u_\Sigma,u_\Sigma\circ\tau_\Sigma)-4\cE_s(u,u\circ \tau_\Sigma)=\cE_s(u-u\circ\tau_\Sigma,u-u\circ\tau_\Sigma)-\cE_s(u_\Sigma-u_\Sigma\circ\tau_\Sigma,u_\Sigma-u_\Sigma\circ\tau_\Sigma)
\]
since~$u+u\circ \tau_\Sigma=u_\Sigma+u_\Sigma\circ \tau_\Sigma$ by definition. Noting that~$u-u\circ\tau_\Sigma$ is anti-symmetric with respect to~$\partial \Sigma$ and that
\[
u_\Sigma-u_\Sigma\circ\tau_\Sigma=|u-u\circ\tau_\Sigma|\1_\Sigma-|u-u\circ\tau_\Sigma|\1_{\R^N\setminus \Sigma}
\]
the assertion follows from Lemma~\ref{lemma:first separation}.
\end{proof}

\begin{theorem}\label{polarization-inequality} 
Let~$s\in(1,3/2)$ and $\Sigma$ an open half-space in $\R^n$- Then for all~$u\in H^s(\R^n)$ we have~$u_\Sigma\in H^s(\R^n)$ and
\[
\cE_s(u_\Sigma,u_\Sigma)\geq \cE_s(u,u).
\]
Moreover, this inequality is strict if~$u_\Sigma\neq u$.
\end{theorem}
\begin{proof}
Recall identity~\eqref{eq:anti-remark}.
Note that
\begin{align*}
\cE_s(u_\Sigma,u_\Sigma)&=\frac{1}{2}\cE_s(u_\Sigma,u_\Sigma)+\frac{1}{2}\cE_s(u_\Sigma\circ\tau_\Sigma,u_\Sigma\circ\tau_\Sigma)\\
&=\frac{1}{2}\cE_s(u_\Sigma+u_\Sigma\circ\tau_\Sigma,u_\Sigma+u_\Sigma\circ\tau_\Sigma)-\cE_s(u_\Sigma,u_\Sigma\circ\tau_\Sigma)\\
&=\frac{1}{2}\cE_s(u+u\circ\tau_\Sigma,u+u\circ\tau_\Sigma)-\cE_s(u_\Sigma,u_\Sigma\circ\tau_\Sigma)\\
&=\cE_s(u,u)+\cE_s(u,u\circ\tau_\Sigma)-\cE_s(u_\Sigma,u_\Sigma\circ\tau_\Sigma).
\end{align*}
Note that we have used~\eqref{eq:anti-remark}.
The claim hence follows from Proposition~\ref{polarization1}.
\end{proof}

\section{The P\'olya-\texorpdfstring{Szeg\H o}{Szego} inequality}

We start by recalling the definition of the symmetric decreasing rearrangement.
We retrieve the definitions in~\cite{talenti}.
\begin{definition}\label{def:schwarz}
Given a measurable function~$u:\R^n\to\R$ the \textit{decreasing rearrangement} of~$u$ is the function
\begin{align*}
\widetilde u:\ [0,\infty)&\longrightarrow[0,\infty) \\
t&\longmapsto\inf\big\{\tau\geq 0:\big|\{|u|>\tau\}\big|<t\big\}
\end{align*}
and the \textit{spherical decreasing rearrangement} is
\begin{align*}
u^*:\ \R^n&\longrightarrow[0,\infty) \\
x&\longmapsto\widetilde u\big(\omega_n|x|^n\big).
\end{align*}
\end{definition}

\subsection{An example}
\label{par:polyaszego-ex}

Consider~$v\in\cH^s_0(B_1)$, nonnegative, radial, and radially decreasing.
Note that we can also write~$v(x)=w(|x|),\ x\in\R^n$, 
for some~$w:[0,\infty)\to[0,\infty)$.

Fix~$x_0\in\R^n,\ |x_0|>2,$ and define
\begin{align*}
u(x):=v(x)+v_{a,x_0}(x),\quad v_{a,x_0}(x):=v(ax-x_0),\qquad x\in\R^n,\ a\geq 1.
\end{align*}
For~$t\in\R$, consider the superlevels of~$u$, \textit{i.e.},~$A_t=\{x\in\R^n:u(x)>t\}$. It holds
\begin{align*}
|A_t|=\big|\{x\in\R^n:u(x)>t\}\big|=\big|\{x\in\R^n:v(x)>t\}\big|+\big|\{x\in\R^n:v(ax)>t\}\big|,
\end{align*}
so that, the symmetrization~$A_t^*$ of~$A_t$ coincides with
\begin{align*}
A_t^*=B\bigg(0,\omega_n^{-1/n}\Big(\big|\{x\in\R^n:v(x)>t\}\big|+\big|\{x\in\R^n:v(ax)>t\}\big|\Big)^{1/n}\bigg).
\end{align*}
We now compute~$u^*$, the spherical decreasing rearrangement of~$u$.
By definition,
\begin{align*}
u^*(y)=\inf\{t\in\R:y\in A_t^*\}
\end{align*}
where
\begin{align*}
y\in A_t^*\qquad\text{if and only if}\qquad
|y|^n 	& < \omega_n^{-1}\big|\{x\in\R^n:v(x)>t\}\big|+
			\omega_n^{-1}\big|\{x\in\R^n:v(ax)>t\}\big| \\
		& = \big|\{z>0:w(z)>t\}\big|^n+\big|\{z>0:w(az)>t\}\big|^n \\ 
		& = \big|w^{-1}(t)\big|^n + \bigg|\frac{w^{-1}(t)}a\bigg|^n = \Big(1+\frac1{a^n}\Big)\big|w^{-1}(t)\big|^n
\end{align*}
which implies 
\begin{align*}
u^*(y)=w\bigg(\frac{a|y|}{{(1+a^n)}^{1/n}}\bigg)
=v_{b,0}(y),\quad b:=\frac{a}{{(1+a^n)}^{1/n}}, \qquad \text{for }y\in\R^n.
\end{align*}
Let us now compare the Sobolev energies:
\begin{align*}
\cE_s(u,u) &= 
\cE_s(v,v)+2\cE_s\big(v,v_{a,x_0}\big)+\cE_s\big(v_{a,x_0},v_{a,x_0}\big) \\
&= 
(1+a^{2s-n})\cE_s(v,v)+2\cE_s\big(v,v_{a,x_0}\big) \geq 
(1+a^{2s-n})\,\cE_s(v,v), \\
\cE_s(u^*,u^*) &= \cE_s\big(v_{b,0},v_{b,0}\big)=b^{2s-n}\cE_s(v,v)=\frac{a^{2s-n}}{{(1+a^n)}^{(2s-n)/n}}\,\cE_s(v,v);
\end{align*}
here we must remark that 
\begin{equation}\label{example-inequality}
\frac{a^{2s-n}}{{(1+a^n)}^{(2s-n)/n}} < 1+a^{2s-n}, \qquad a> 1,
\end{equation}
so that we find
\begin{equation}\label{example1}
\cE_s(u^*,u^*)< \cE_s(u,u).
\end{equation}

\noindent Indeed, \eqref{example-inequality} is equivalent to
\begin{align*}
{(1+a^n)}^{1-2s/n} < 1+{(a^n)}^{1-2s/n}, \qquad a> 1,
\end{align*}
which holds by the subadditivity of~$t\mapsto t^\gamma$ for~$t\geq 1,\gamma\leq 1$.

\subsection{A counterexample}
\label{par:polyaszego-counter}

Let us consider~$n=1$, restrict~$s\in(1,3/2)$, and define for~$M>1$
\begin{align*}
u(x):=\left\lbrace\begin{aligned}
& 0 & & \text{for } x\leq -2 \\
& x+2 & & \text{for }-2<x\leq 0 \\
& -x+2 & & \text{for } 0<x\leq 1 \\
& 1 & & \text{for }1<x\leq 2M-1 \\
& -x+2M & & \text{for }2M-1<x\leq 2M \\
& 0 & & \text{for } x>2M.
\end{aligned}\right.
\end{align*}
The spherical rearrangement of~$u$ is
\begin{align*}
u^*(x):=\left\lbrace\begin{aligned}
& -|x|+2 & & \text{for } 0\leq|x|\leq 1 \\
& 1 & & \text{for } 1<|x|\leq M \\
& -|x|+M+1 & & \text{for } M<|x|\leq M+1 \\
& 0 & & \text{for } |x|> M+1.
\end{aligned}\right.
\end{align*}
\begin{figure}[!t] \centering

	\begin{tikzpicture}\footnotesize
		\draw (-.49\textwidth,0pt) -- (-.05\textwidth,0pt) ;
		\node at (-.05\textwidth-3pt,6pt) {$x$} ;
		\node at (-.345\textwidth+10pt,.15\textwidth-3pt) {$u(x)$} ;
		\draw (-.35\textwidth,-8pt) -- (-.35\textwidth, .15\textwidth) ;
		\draw (.05\textwidth,0pt) -- (.49\textwidth,0pt) ;
		\node at (.5\textwidth-6pt,6pt) {$x$} ;
		\node at (.28\textwidth+12pt,.15\textwidth-3pt) {$u^*(x)$} ;
		\draw (.275\textwidth,-8pt) -- (.275\textwidth, .15\textwidth) ;
		\node at (-.45\textwidth,-6pt) {$-2$} ;
		\node at (-.35\textwidth-5pt,-6pt) {$0$} ;
		\node at (-.3\textwidth,-6pt) {$1$} ;
		\draw [densely dotted] (-.3\textwidth,0pt) -- (-.3\textwidth,.05\textwidth) ;
		\node at (-.16\textwidth,-6pt) {$2M\!-\!1$} ;
		\draw [densely dotted] (-.15\textwidth,0pt) -- (-.15\textwidth,.05\textwidth) ;
		\node at (-.09\textwidth,-6pt) {$2M$} ;
		\node at (-.35\textwidth-5pt,.05\textwidth) {$1$} ;
		\draw [densely dotted] (-.35\textwidth,.05\textwidth) -- (-.3\textwidth,.05\textwidth) ;
		\node at (-.35\textwidth-5pt,.1\textwidth+3pt) {$2$} ;
		\draw [ultra thick] plot coordinates {(-.49\textwidth,0pt) (-.45\textwidth,0pt) (-.35\textwidth,.1\textwidth) (-.3\textwidth,.05\textwidth) (-.15\textwidth,.05\textwidth) (-.1\textwidth,0pt) (-.05\textwidth,0pt)} ;
		\node at (.09\textwidth,-6pt) {$-\!M\!-\!1$} ;
		\node at (.16\textwidth,-6pt) {$-\!M$} ;
		\node at (.225\textwidth,-6pt) {$-1$} ;
		\node at (.275\textwidth-5pt,-6pt) {$0$} ;
		\node at (.325\textwidth,-6pt) {$1$} ;
		\node at (.4\textwidth,-6pt) {$M$} ;
		\node at (.45\textwidth,-6pt) {$M\!+\!1$} ;
		\draw [densely dotted] (.225\textwidth,0pt) -- (.225\textwidth,.05\textwidth) ;
		\draw [densely dotted] (.325\textwidth,0pt) -- (.325\textwidth,.05\textwidth) ;
		\draw [densely dotted] (.15\textwidth,0pt) -- (.15\textwidth,.05\textwidth) ;
		\draw [densely dotted] (.4\textwidth,0pt) -- (.4\textwidth,.05\textwidth) ;
		\draw [densely dotted] (.225\textwidth,.05\textwidth) -- (.325\textwidth,.05\textwidth) ;
		\node at (.275\textwidth-5pt,.05\textwidth-6pt) {$1$} ;		
		\node at (.275\textwidth-5pt,.1\textwidth+3pt) {$2$} ;
		\draw [ultra thick] plot coordinates {(.49\textwidth,0pt) (.45\textwidth,0pt) (.4\textwidth,.05\textwidth) (.325\textwidth,.05\textwidth) (.275\textwidth,.1\textwidth) (.225\textwidth,.05\textwidth) (.15\textwidth,.05\textwidth) (.1\textwidth,0pt) (.05\textwidth,0pt)} ;
	\end{tikzpicture}
	\caption{The graphs of~$u$ and~$u^*$.}
\end{figure}
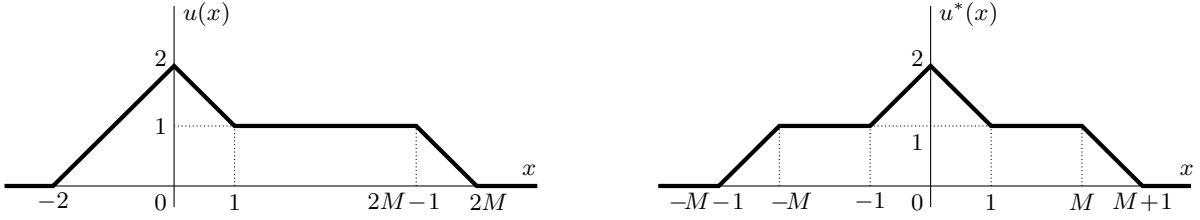

Let us now explicitly compute the Dirichlet energies associated to both functions. 
As to~$u$, we have (we are going to extensively use translation invariance and~\eqref{eq:energies})
\begin{align*}
&\cE_s(u,u) 
=\cE_{s-1}(u',u')=\cE_{s-1}\big( \1_{(-2,0)}-\1_{(0,1)}-\1_{(2M-1,2M)},\1_{(-2,0)}-\1_{(0,1)}-\1_{(2M-1,2M)} \big) \\
& =\cE_{s-1}\big( \1_{(-2,-1)}+\1_{(-1,0)}-\1_{(0,1)}-\1_{(2M-1,2M)},\1_{(-2,-1)}+\1_{(-1,0)}-\1_{(0,1)}-\1_{(2M-1,2M)} \big) \\
&=4\cE_{s-1}(\1_{(0,1)},\1_{(0,1)})-2\cE_{s-1}(\1_{(-2,-1)},\1_{(0,1)})-2\cE_s(\1_{(-2,-1)},\1_{(2M-1,2M)}) \\
&\qquad -2\cE_{s-1}(\1_{(-1,0)},\1_{(2M-1,2M)})+2\cE_{s-1}(\1_{(0,1)},\1_{(2M-1,2M)}),
\end{align*}
whereas
\begin{align*}
&\cE_s(u^*,u^*)=\\
&=\cE_{s-1}(\1_{(-M-1,-M)}+\1_{(-1,0)}-\1_{(0,1)}-\1_{(M,M+1)},\1_{(-M-1,-M)}+\1_{(-1,0)}-\1_{(0,1)}-\1_{(M,M+1)}) \\
&=4\cE_{s-1}(\1_{(0,1)},\1_{(0,1)})+4\cE_{s-1}(\1_{(0,1)},\1_{(M,M+1)})-4\cE_{s-1}(\1_{(-1,0)},\1_{(M,M+1)}) \\
&\qquad -2\cE_{s-1}(\1_{(-M-1,-M)},\1_{(M,M+1)})-2\cE_{s-1}(\1_{(-1,0)},\1_{(0,1)}).
\end{align*}
Summing up,
\begin{align*}
\cE_s(u,u)-\cE_s(u^*,u^*)
=\ & -2\cE_{s-1}(\1_{(-2,-1)},\1_{(0,1)})-2\cE_s(\1_{(-2,-1)},\1_{(2M-1,2M)}) \\
   & -2\cE_{s-1}(\1_{(-1,0)},\1_{(2M-1,2M)})+2\cE_{s-1}(\1_{(0,1)},\1_{(2M-1,2M)}) \\
   & -4\cE_{s-1}(\1_{(0,1)},\1_{(M,M+1)})+4\cE_{s-1}(\1_{(-1,0)},\1_{(M,M+1)}) \\
   & +2\cE_{s-1}(\1_{(-M-1,-M)},\1_{(M,M+1)})+2\cE_{s-1}(\1_{(-1,0)},\1_{(0,1)}) \\
=\ & -2\cE_{s-1}(\1_{(0,1)},\1_{(-2,-1)})-2\cE_s(\1_{(0,1)},\1_{(2M+1,2M+2)}) \\
   & -2\cE_{s-1}(\1_{(0,1)},\1_{(2M,2M+1)})+2\cE_{s-1}(\1_{(0,1)},\1_{(2M-1,2M)}) \\
   & -4\cE_{s-1}(\1_{(0,1)},\1_{(M,M+1)})+4\cE_{s-1}(\1_{(0,1)},\1_{(M+1,M+2)}) \\
   & +2\cE_{s-1}(\1_{(0,1)},\1_{(2M+1,2M+2)})+2\cE_{s-1}(\1_{(0,1)},\1_{(-1,0)}) \\
=\ & -2\cE_{s-1}(\1_{(0,1)},\1_{(-2,-1)})+2\cE_{s-1}(\1_{(0,1)},\1_{(-1,0)}) \\
   & -2\cE_{s-1}(\1_{(0,1)},\1_{(2M,2M+1)})+2\cE_{s-1}(\1_{(0,1)},\1_{(2M-1,2M)}) \\
   & -4\cE_{s-1}(\1_{(0,1)},\1_{(M,M+1)})+4\cE_{s-1}(\1_{(0,1)},\1_{(M+1,M+2)}).
\end{align*}
As~$\cE_{s-1}(\1_{(0,1)},\1_{(-1,0)})<\cE_{s-1}(\1_{(0,1)},\1_{(-2,-1)})$ 
---because the kernel is radially decreasing and the interactions are negative by~\eqref{disjointsupport}---
and~$\cE_{s-1}(\1_{(0,1)},\1_{(M,M+1)})\to 0$ as~$M\uparrow\infty$
---along with the other similar terms---, we deduce that there exists~$M^*(s)>0$ 
such that for any~$M>M^*(s)$ 
\begin{equation}\label{example2}
\cE_s(u,u)<\cE_s(u^*,u^*).
\end{equation}

\begin{remark}
A close inspection of~$\cE_s(u,u)-\cE_s(u^*,u^*)$ can actually show that
$M^*(s)=1$ for any~$s\in(1,3/2)$. Indeed, let us define the functions
\begin{align*}
\begin{aligned}
f:[1,\infty) & \longrightarrow \R \\
r & \longmapsto -2\cE_{s-1}(\1_{(0,1)},\1_{(r,r+1)}) ,
\end{aligned}
\qquad\text{and}\qquad
\begin{aligned}
g:[1,\infty) & \longrightarrow \R \\
r & \longmapsto f(r+1)-f(r).
\end{aligned}
\end{align*}
As~$f$ is positive and decreasing,~$g$ is negative and 
\begin{align*}
\cE_s(u,u)-\cE_s(u^*,u^*)=g(1)-2g(M)+g(2M-1).
\end{align*}
Note that the following identities hold for~$r>1$:
\begin{align*}
f(r) &= c_{1,s-1}\int_0^1\int_0^1\big(y-x+r\big)^{ 1-2s}\;dy\;dx \\
f'''(r) &= c_{1,s-1}( 1-2s)(-2s)(-1-2s)\int_0^1\int_0^1\big(y-x+r\big)^{-2-2s}\;dy\;dx<0
\end{align*}
so that
\begin{align*}
g''(r)=f''(r+1)-f''(r)<0
\qquad\text{for }r>1.
\end{align*}
Using twice a second order Taylor expansion centered at~$M$ with the remainder in Lagrange form, 
one deduces
\begin{align*}
\cE_s(u,u)-\cE_s(u^*,u^*)=g(1)-2g(M)+g(2M-1)=g''(\xi)(M-1)^2<0.
\end{align*}
\end{remark}


\begin{remark}
In view of \eqref{example1} and \eqref{example2}, it follows immediately that a P\'olya-Szeg\H o inequality or a reversed version of it cannot hold.
\end{remark}

\section{Analysis of the two-balls domain}

In this section, let~$D_1,D_2\subset\R^n$ be two open balls with~$D_1\cap D_2=\emptyset$ and~$\Omega=D_1\cup D_2$.
In this section we perform a spectral analysis of the operator~$\Ds$ on the space~$\cH^s_0(\Omega)$, by showing what announced in the Introduction, \textit{i.e.}, that the first eigenvalue is positive and simple, with an associated eigenfunction which is possibly sign-changing according to the value of~$s$: this is done in Subsection~\ref{sub:firsteigenfunction}. Moreover, we improve the analysis carried out in~\cite{ajs-maxprinc} and~\cite{ajs-repr} by giving sufficient conditions under which maximum principles are recovered on~$\Omega$: this is done in Subsection~\ref{sub:recovery}.

\subsection{The first eigenfunction}
\label{sub:firsteigenfunction}

\begin{lemma}
For $U\subset \R^n$ open and bounded, denote by
$$
\lambda(U):=\min_{\substack{u\in \cH^s_0(U)\\ u\neq 0}} \frac{\cE_s(u,u)}{\|u\|_{L^2(U)}^2}
$$
the first Dirichlet eigenvalue of $(-\Delta)^s$ in $U$. Then
$\lambda(U)>0$ and it is attained by some~$\varphi\in\cH^s_0(U)$.
\end{lemma}

The proof of this lemma is standard and we omit it here.

In the following, we denote by~$\pi^{\pm}_i$,~$i=1,2$, the projections of~$\cH^s_0(D_i)$ onto their positive and negative cones, as defined and analysed in Appendix~\ref{sec:proj}.

\begin{proposition}\label{proj-vs-orig} 
Let~$\lfloor s\rfloor \in 2\N_0+1$.
The function~$\tilde\varphi=2\pi_1^+\varphi_1-\varphi_1+2\pi_2^-\varphi_2-\varphi_2$ satisfies
\begin{align}\label{better}
\left\|\tilde\varphi\right\|_{L^2(\Omega)}\ \geq\ \left\|\varphi\right\|_{L^2(\Omega)},\qquad
\cE_s(\tilde\varphi,\tilde\varphi)\ \leq\ \cE_s(\varphi,\varphi).
\end{align}
If instead~$\lfloor s\rfloor \in 2\N_0$, then~$\tilde\varphi=2\pi_1^+\varphi_1-\varphi_1+2\pi_2^+\varphi_2-\varphi_2$
satisfies~\eqref{better}.
\end{proposition}
\begin{proof}
Let us suppose~$\lfloor s\rfloor \in 2\N_0+1$, the case~$\lfloor s\rfloor \in 2\N_0$ will follow under suitable minor modifications.

By the maximum principle on balls~\cite{ajs-boggio}*{Theorem~1.1} and~\eqref{eq:ppp}, it holds 
$\pi_1^+\varphi_1\geq\varphi_1$ and~$\pi_2^-\varphi_2\leq\varphi_2$. One has
\[
|\varphi|\leq|\varphi_1-\pi_1^+\varphi_1+\pi_1^+\varphi_1|+|\varphi_2-\pi_2^-\varphi_2+\pi_2^-\varphi_2|
= 2\pi_1^+\varphi_1-\varphi_1-2\pi_2^-\varphi_2+\varphi_2=|\tilde\varphi|
\]
and this proves the claim about the~$L^2$ norm. Also,
\begin{align}
\cE_s(\tilde\varphi,\tilde\varphi)
&=\cE_s(2\pi_1^+\varphi_1-\varphi_1+2\pi_2^-\varphi_2-\varphi_2,2\pi_1^+\varphi_1-\varphi_1+2\pi_2^-\varphi_2-\varphi_2)  \nonumber \\ 
& =\cE_s(2\pi_1^+\varphi_1-\varphi_1,2\pi_1^+\varphi_1-\varphi_1)+\cE_s(2\pi_2^-\varphi_2-\varphi_2,2\pi_2^-\varphi_2-\varphi_2) \nonumber \\
&\qquad +2\cE_s(2\pi_1^+\varphi_1-\varphi_1,2\pi_2^-\varphi_2-\varphi_2)  \nonumber \\
& \leq \cE_s(\varphi_1,\varphi_1)+\cE_s(\varphi_2,\varphi_2)-2\cE_s(\varphi_1,\varphi_2) \nonumber \\
&\qquad +4\cE_s(\pi^+_1\varphi_1-\varphi_1,\pi^-_2\varphi_2)+4\cE_s(\pi^+_1\varphi_1,\pi^-_2\varphi_2-\varphi_2) \label{8632263} \\
& \leq \cE_s(\varphi,\varphi)  \nonumber
\end{align}
where we have used~\eqref{eq:downward2} and the inner product of functions with segregated supports~\eqref{disjointsupport}.
\end{proof}

\begin{lemma}\label{lem:nontriv} 
$\lambda(\Omega)<\min\{\lambda(D_1),\lambda(D_2)\}$.
As a consequence,~$\varphi\not\equiv0$ in~$D_1$ and~$\varphi\not\equiv0$ in~$D_2$.
\end{lemma}
\begin{proof}
Suppose, without loss of generality, that~$\lambda(D_1)\leq\lambda(D_2)$.
Denote by~$u_1\in\cH^s_0(D_1)$ and~$u_2\in\cH^s(D_2)$ the normalized (positive) 
eigenfunctions associated to~$\Ds$
on~$D_1$ and~$D_2$ respectively. For any~$\alpha\in[0,1]$:
\begin{multline*}
\cE_s\Big(\sqrt{1-\alpha^2}\,u_1+(-1)^{\intgr s}\alpha u_2,\sqrt{1-\alpha^2}\,u_1+(-1)^{\intgr s}\alpha u_2\Big)=\\
=\big(1-\alpha^2\big)\lambda(D_1)+\alpha^2\lambda(D_2)+2(-1)^{\intgr s}\alpha\sqrt{1-\alpha^2}\,\cE_s(u_1,u_2).
\end{multline*}
When~$\alpha>0$ is very small, one has~$\alpha^2\lambda(D_2)+2(-1)^{\intgr s}\alpha\sqrt{1-\alpha^2}\,\cE_s(u_1,u_2)<0$
and therefore there is a suitable choice~$\alpha_*\in(0,1)$ such that
\begin{align*}
\cE_s\Big(\sqrt{1-\alpha_*^2}\,u_1-\alpha_*u_2,\sqrt{1-\alpha_*^2}\,u_1-\alpha_*u_2\Big)
<\lambda(D_1)\leq\lambda(D_2)
\end{align*}
which is proving the first claim.

If it were~$\varphi\equiv 0$ in~$D_2$ then it would also hold
\begin{align*}
\varphi\in\cH^s_0(D_1),\qquad\cE_s(\varphi,\varphi)=\lambda(\Omega)<\lambda(D_1),
\end{align*}
which contradicts the definition of~$\lambda(D_1)$.
\end{proof}

\begin{theorem}\label{thm:first-twoball}
$\lambda(\Omega)$ is simple.

If~$\lfloor s\rfloor\in 2\N_0+1$, then~$\varphi$ is of one sign on~$D_1$ and of the opposite one in~$D_2$.

If instead~$\lfloor s\rfloor\in 2\N_0$, then~$\varphi$ is of constant sign in~$\Omega$.
\end{theorem}
\begin{proof}
Let us suppose~$\lfloor s\rfloor \in 2\N_0+1$, the case~$\lfloor s\rfloor \in 2\N_0$ will follow after suitable minor modifications.

In the notations of Proposition~\ref{proj-vs-orig}, we already know that 
\begin{align*}
\cE_s(\tilde\varphi,\tilde\varphi)\leq\cE_s(\varphi,\varphi)
\end{align*}
and, in view of~\eqref{8632263}, equality can hold only if
\[
\cE_s(\pi_1^+\varphi_1,\pi_2^-\varphi_2-\varphi_2)+\cE_s(\pi_1^+\varphi_1-\varphi_1,\pi_2^-\varphi_2)=0.
\]
This last condition in turn holds if and only if one of the following is true:
\begin{enumerate}[label=\it\roman*)]
\item ~$\pi_1^+\varphi_1=0$ and~$\pi_2^-\varphi_2=0$, 
which means that~$\varphi_1\leq 0$ and~$\varphi_2\geq 0$;
\item ~$\pi_2^-\varphi_2=\varphi_2$ and~$\pi_1^+\varphi_1=\varphi_1$, 
which means that~$\varphi_2\leq 0$ and~$\varphi_1\geq 0$;
\item ~$\varphi_1=0$ or~$\varphi_2=0$, but in this case the minimizing condition would be broken, 
see Lemma~\ref{lem:nontriv}. 
\end{enumerate}
This proves the claim about the sign of the eigenfunction.

Suppose now that we have two minimizers~$\varphi$ and~$\psi$. 
Without loss of generality (by the previous claim, they are necessarily of one sign on each ball) 
we can suppose that~$\varphi\psi\leq 0$ in~$\Omega$.
Define, for any~$t\in[0,1]$,~$w_t:=t\varphi+\sqrt{1-t^2}\psi$. 
One can verify that~$w_t$ is either still a minimizer or trivial.
In the first case,~$w_t$ must be of one sign on each ball and, 
since~$w_1=\varphi$ and~$w_0=\psi$ are of opposite sign
throughout~$\Omega$, there must exist a~$t_1\in(0,1)$ such that~$w_{t_1}=0$ in~$D_1$. 
This feature is not compatible with minimality (compare with {\it iii)} above) 
so we must conclude~$w_{t_1}=0$ in~$\Omega$.
Thus~$\varphi$ and~$\psi$ are not linearly independent.
\end{proof}

\subsection{Partial recovery of maximum principles}
\label{sub:recovery}

%
%
%
%

\begin{proposition}\label{part-max-princ-prelim}
Suppose~$|D_1|=|D_2|=\omega_n$ and~$\dist(D_1,D_2)\geq 2$.
Denote by~$\tau:\R^n\to\R^n$ the inversion of~$\R^n$ such that~$\tau(D_1)=D_2$.
Let~$g\in L^2(D_1),\ g\geq 0,$ and~$f\in L^2(\Omega),\ f=g-(-1)^{\lfloor s\rfloor}g\circ\tau$;
then the weak solution~$u\in\cH^s_0(\Omega)$ of 
\begin{align*}
\Ds u=f\quad\text{ in }\ \Omega
\end{align*}
is of the form~$u=v-(-1)^{\lfloor s\rfloor}v\circ\tau,\ v\in\cH^s_0(D_1)$, with~$v\geq 0$ in~$D_1$.
\end{proposition}
\begin{proof}
The solution~$u\in\cH^s_0(\Omega)$ is given by the Green representation
\[
u(x)=\int_\Omega G_\Omega(x,y)\,f(y)\;dy,
\qquad x\in\Omega.
\]
Using the symmetry of~$f$,~$u$ can be alternatively written
\begin{align*}
u(x)=\int_{D_1}\big[G_\Omega(x,y)-(-1)^{\lfloor s\rfloor}G_\Omega(x,\tau(y))\big]g(y)\;dy ,
\qquad x\in\Omega.
\end{align*}
We claim that
\begin{equation}\label{claim-green}
G_\Omega(x,y)-(-1)^{\lfloor s\rfloor}G_\Omega(x,\tau(y))\ \geq\ 0,\qquad x,y\in D_1.
\end{equation}
The Green function~$G_\Omega(x,\cdot)$ can be represented, for~$x\in D_1$,
by making use of the Poisson kernel (see~\cite{ajs-repr}*{Proposition 3.4}) as
\begin{align*}
G_\Omega(x,y)\ =\ \left\lbrace\begin{aligned}
& G_{D_1}(x,y)+\int_{D_2}\Gamma_{D_1}(y,z)\,G_\Omega(x,z)\;dz & y\in D_1, \\
& \int_{D_1}\Gamma_{D_2}(y,z)\,G_\Omega(x,z)\;dz
=\int_{D_2}\Gamma_{D_1}(\tau(y),z)\,G_\Omega(x,\tau(z))\;dz & y\in D_2,
\end{aligned}\right.
\end{align*}
and thus
\begin{align*}
& G_\Omega(x,y)-(-1)^{\lfloor s\rfloor}G_\Omega(x,\tau(y))= \\
& =G_{D_1}(x,y)+\int_{D_2}\Gamma_{D_1}(y,z)\,G_\Omega(x,z)\;dz
-(-1)^{\lfloor s\rfloor}\int_{D_2}\Gamma_{D_1}(y,z)\,G_\Omega(x,\tau(z))\;dz
\qquad x,y\in D_1.
\end{align*}
Denoting by~$w_x(y):=G_\Omega(x,y)-(-1)^{\lfloor s\rfloor}G_\Omega(x,\tau(y)),\ x\in D_1,\ y\in\Omega$, we have for~$y\in D_1$
\begin{multline*}
w_x(y)= G_{D_1}(x,y)+\int_{D_2}\Gamma_{D_1}(y,z)\,w_x(z)\;dz= \\
= G_{D_1}(x,y)+\int_{D_2}\Gamma_{D_1}(y,z)\,G_{D_1}(x,\tau(z))\;dz 
+\int_{D_2}\bigg(\int_{D_2}\Gamma_{D_1}(y,z)\Gamma_{D_1}(z,\xi)\;dz\bigg)w_x(\xi)\;d\xi.
\end{multline*}
The rest of the proof consists of two different steps:
the first one is proving
\begin{equation}\label{claim-green2}
G_{D_1}(x,y)+\int_{D_2}\Gamma_{D_1}(y,z)\,G_{D_1}(x,\tau(z))\;dz \geq 0
\qquad x,y\in D_1,
\end{equation}
when~$\dist(D_1,D_2)\geq 2$;
once this done, the second one is bootstrapping the resulting inequality
\begin{equation}\label{claim-green3}
w_x(y) \geq \int_{D_2}\bigg(\int_{D_2}\Gamma_{D_1}(y,z)\Gamma_{D_1}(z,\xi)\;dz\bigg)w_x(\xi)\;d\xi
\end{equation}
to get~\eqref{claim-green} in the end, as desired.

As to~\eqref{claim-green2}, we use the explicit expression of~$G_{D_1}$
known as {\it Boggio's formula} (see~\cite{ajs-boggio}*{Theorem 1.1},~\cite{dg}*{Theorem 1}, or also~\cite{dkk}*{Remark 1}),
\begin{align*}
G_{D_1}(x,y)\ &=\ k_{n,s}\frac{\big(1-|x|^2\big)^s\big(1-|y|^2\big)^s}{{|x-y|}^{n}}\int_0^1
\frac{\eta^{s-1}}{\big(\rho(x,y)\eta+1\big)^{n/2}}\;d\eta, \qquad x,y\in D_1, \\
\rho(x,y)&=\frac{\left(1-|x|^2\right)\left(1-|y|^2\right)}{{|x-y|}^2}, \\
k_{n,s}&=\frac{\Gamma(n/2)}{2^{2s}\pi^{n/2}\Gamma(s)^2},
\end{align*}
and the explicit formula for the Poisson kernel~$\Gamma_{D_1}$ (see~\cite{ajs-repr}*{Theorem 1.1})
\begin{align*}
\Gamma_{D_1}(y,z)&=\frac{(-1)^{\lfloor s\rfloor}\gamma_{n,s}}{{|y-z|}^n}\bigg(\frac{1-|y|^2}{|z|^2-1}\bigg)^s,
\qquad y\in D_1,\ z\in\R^n\setminus\overline D_1, \\
\gamma_{n,s}&=\frac{\Gamma(n/2)}{\pi^{n/2}\Gamma(s-\lfloor s\rfloor)\,\Gamma(1-s+\lfloor s\rfloor)}.
\end{align*}
In order to get~\eqref{claim-green2} we estimate
\begin{multline*}
k_{n,s}\,\frac{\big(1-|x|^2\big)^s}{{|x-y|}^{n}}\int_0^1
\frac{\eta^{s-1}}{\big(\rho(x,y)\eta+1\big)^{n/2}}\;d\eta
-\gamma_{n,s}\int_{D_2}\frac{G_{D_1}(x,\tau(z))}{{|y-z|}^n{(|z|^2-1)}^s}\;dz\geq\\
\geq \frac{k_{n,s}}s\frac{\big(1-|x|^2\big)^s}{\big((1-|x|^2)(1-|y|^2)+{|x-y|}^2\big)^{n/2}}
-\frac{\gamma_{n,s}}{2^n}\int_{D_2}G_{D_1}(x,\tau(z))\;dz 
\end{multline*}
where we have used~$\dist(D_1,D_2)\geq 2$. Next,
\begin{align*}
\int_{D_2}G_{D_1}(x,\tau(z))\;dz=\int_{D_1}G_{D_1}(x,z)\;dz
\end{align*}
gives rise to the torsion function of~$D_1$ (see~\cite{MR3294242}*{Lemma 2.1})
\begin{align*}
\int_{D_1}G_{D_1}(x,z)\;dz=
\frac{\Gamma(n/2)}{2^{2s}\Gamma(n/2+s)\Gamma(1+s)}\,\big(1-|x|^2\big)^s
\end{align*}
so that
\begin{multline*}
k_{n,s}\,\frac{\big(1-|x|^2\big)^s}{{|x-y|}^{n}}\int_0^1
\frac{\eta^{s-1}}{\big(\rho(x,y)\eta+1\big)^{n/2}}\;d\eta
-\gamma_{n,s}\int_{D_2}\frac{G_{D_1}(x,\tau(z))}{{|y-z|}^n{(|z|^2-1)}^s}\;dz\geq\\
\geq \frac{k_{n,s}}{s2^n}\big(1-|x|^2\big)^s
-\frac{\gamma_{n,s}}{2^n}\frac{\Gamma(n/2)}{2^{2s}\Gamma(n/2+s)\Gamma(1+s)}\,\big(1-|x|^2\big)^s
\end{multline*}
where we have used~$\dist(D_1,D_2)\geq 2$ and~$(1-|x|^2)(1-|y|^2)+|x-y|^2\leq 4$ for~$x,y\in D_1$. The above estimate provides~\eqref{claim-green2} if we check
\begin{align*}
\frac{1}{s\Gamma(s)^2}
-\frac{1}{\Gamma(s-\lfloor s\rfloor)\,\Gamma(1-s+\lfloor s\rfloor)}\frac{\Gamma(n/2)}{\Gamma(n/2+s)\Gamma(1+s)}\geq 0
\end{align*}
which, by using standard properties\footnote{We need in particular the \textit{Euler's reflection formula}:~$\Gamma(\alpha)\Gamma(1-\alpha)=\dfrac{\pi}{\sin(\pi \alpha)}$ for~$\alpha\in\mathbb{C}\setminus\Z$.} of the~$\Gamma$ function, is equivalent to
\begin{align}\label{8743297423}
\frac{\sin\big((s-\lfloor s\rfloor)\pi\big)}{\pi}\frac{\Gamma(n/2)\Gamma(s)}{\Gamma(n/2+s)}\leq 1.
\end{align}
We now use the integral representation of the \textit{Beta function} in order to deduce
\begin{align*}
\frac{\Gamma(n/2)\Gamma(s)}{\Gamma(n/2+s)}=\int_0^1{r}^{n/2-1}{(1-r)}^{s-1}\;dr\leq\int_0^1{r}^{n/2-1}\;dr=\frac2n\leq2.
\end{align*}
This proves~\eqref{8743297423} and, therefore,~\eqref{claim-green2}.

From the above,~\eqref{claim-green3} follows and we can proceed as in~\cite{ajs-repr}
to conclude the proof.
\end{proof}

\begin{theorem}
Under the assumptions of Proposition~\ref{part-max-princ-prelim}, the weak solution~$v\in\cH^s_0(D_1)$ of
\begin{align*}
\Ds v-{(-1)}^{\lfloor s\rfloor}\Ds(v\circ\tau)=g
\qquad\text{in }D_1
\end{align*}
is nonnegative.
\end{theorem}
\begin{proof}
It follows immediately from Proposition~\ref{part-max-princ-prelim}.
\end{proof}

\section{The Faber-Krahn inequality in dimension one}

In this section we restrict to~$n=1$ and consider an open set~$\Omega\subset\R$ with~$|\Omega|=2$.
We denote by
\begin{align*}
\Omega &= \bigcup_{k\in \cM} I_k,
\qquad\text{where~$I_k$'s are disjoint open intervals,} \ \cM\subseteq \N, \\
\varphi &= \sum_{k\in\cM} \varphi_k,
\qquad \varphi_k\in\cH^s_0(I_k), \text{ for any } k\in\cM.
\end{align*}

\begin{lemma}\label{lem:one sign}
For any~$s>0$ and~$k\in\cM$, any first eigenfunction~$\varphi$ is of one sign in~$I_k$.
\end{lemma}
\begin{proof}
Denote by~$\pi^+_k$ the projection onto the positive cone~$\cC^+(I_k)$
as defined in~\eqref{positive cone}-\eqref{projection property}. 
Consider the function
\begin{align*}
\widetilde\varphi=\varphi-2\varphi_k+2\pi^+_k\varphi_k \quad\in\cH^s_0(\Omega)
\end{align*}
and compute
\begin{align*}
\|\widetilde\varphi\|_{L^2(\Omega)}=
\|\varphi-\varphi_k\|_{L^2(\Omega\setminus I_k)}+\|2\pi^+_1\varphi_k-\varphi_k\|_{L^2(I_k)}
\end{align*}
where~$|\varphi_k|=|\varphi_k-\pi^+_k\varphi_k+\pi^+_k\varphi_k|\leq 2\pi^+_k\varphi_k-\varphi_k$
by the maximum principles on intervals and~\eqref{eq:ppp}. Therefore~$\|\widetilde\varphi\|_{L^2(\Omega)}\leq\|\varphi\|_{L^2(\Omega)}$.

Furthermore
\begin{align*}
\cE_s(\widetilde\varphi,\widetilde\varphi) & =
\cE_s(\varphi-\varphi_k,\varphi-\varphi_k)
+\cE_s(2\pi_k^+\varphi_k-\varphi_k,2\pi_k^+\varphi_k-\varphi_k)
+2\cE_s(2\pi_k^+\varphi_k-\varphi_k,\varphi-\varphi_k) \\
& \leq
\cE_s(\varphi-\varphi_k,\varphi-\varphi_k)
+\cE_s(\varphi_k,\varphi_k)
+2\cE_s(2\pi_k^+\varphi_k-\varphi_k,\varphi-\varphi_k) \\
& =
\cE_s(\varphi,\varphi)
+4\cE_s(\pi_k^+\varphi_k-\varphi_k,\varphi-\varphi_k)
\end{align*}
by~\eqref{eq:downward2}. Up to changing the sign of~$\varphi-\varphi_k$, we deduce~$\cE_s(\widetilde\varphi,\widetilde\varphi)\leq\cE_s(\varphi,\varphi)$ by~\eqref{disjointsupport}. 
\end{proof}

\begin{lemma}\label{lem:fill the holes}
For any~$s>0$, if there exist~$k,h\in\cM,k\neq h$, such that~$\mathrm{dist}(I_k,I_h)=0$,
then there exists~$x_0\in\R$ such that~$I_k\cup\{x_0\}\cup I_h$ is connected
and~$\lambda(\Omega\cup\{x_0\})\leq\lambda(\Omega)$.
\end{lemma}
\begin{proof}
This simply follows by the inclusion~$\cH^s_0(\Omega)\subseteq\cH^s_0(\Omega\cup\{x_0\})$.
\end{proof}

\subsection{When the integer part of \texorpdfstring{$s$}{s} is even}

\begin{proposition}\label{prop:even positive}
Let~$\lfloor s\rfloor\in 2\N_0$. Then~$\varphi\geq 0$ in~$\Omega$. 
\end{proposition}
\begin{proof}
By Lemma~\ref{lem:one sign} we already know that~$\varphi$ is of one sign on each connected component of~$\Omega$, so that~$\varphi^+,\varphi^-\in\cH^s_0(\Omega)$. As~$\lfloor s\rfloor\in 2\N_0$ by assumption, one has by~\eqref{disjointsupport}
\begin{align*}
\cE_s(|\varphi|,|\varphi|)&=
\cE_s(\varphi^+,\varphi^+)+\cE_s(\varphi^-,\varphi^-)+2\cE_s(\varphi^+,\varphi^-)\\
&<\cE_s(\varphi^+,\varphi^+)+\cE_s(\varphi^-,\varphi^-)-2\cE_s(\varphi^+,\varphi^-)
=\cE_s(\varphi,\varphi).
\end{align*}
\end{proof}

\begin{theorem}\label{thm:faber-krahn even}
Let~$\lfloor s\rfloor\in 2\N_0$. 
Then~$\lambda(-1,1)\leq\lambda(\Omega)$ for any open set~$\Omega\subset\R$ with~$|\Omega|=2$.
Moreover, equality holds if and only if~$\Omega$ is an interval.
\end{theorem}
\begin{proof}
By Lemma~\ref{lem:fill the holes}, we can restrict our attention to those~$\Omega$'s
satisfying~$\dist(I_k,I_h)>0$ for all~$k,h\in\cM$,~$k\neq h$.

Let~$x_0\in\R\setminus\Omega$ such that~$(-\infty,x_0)\cap\Omega\neq\emptyset\neq\Omega\cap(x_0,+\infty)$.
Write~$\varphi=\varphi\1_{(-\infty,x_0)}+\varphi\1_{(x_0,+\infty)}$.
We have
\begin{align*}
\cE_s(\varphi,\varphi)=
\cE_s(\varphi\1_{(-\infty,x_0)},\varphi\1_{(-\infty,x_0)})
+\cE_s(\varphi\1_{(x_0,+\infty)},\varphi\1_{(x_0,+\infty)})
+2\cE_s(\varphi\1_{(-\infty,x_0)},\varphi\1_{(x_0,+\infty)}).
\end{align*}
By Proposition~\ref{prop:even positive} and the fact that~$\lfloor s\rfloor\in 2\N_0$,
the last term is non-positive, indeed
\begin{align*}
\cE_s(\varphi\1_{(-\infty,x_0)},\varphi\1_{(x_0,+\infty)})=-\frac{c_{1,s}}2\int_{-\infty}^{x_0}\int_{x_0}^{+\infty}\frac{\varphi(x)\,\varphi(y)}{{(y-x)}^{1+2s}}\;dy\;dx.
\end{align*}
Since there exists~$\eps>0$ such that~$(x_0-\eps,x_0)\cap\Omega=\emptyset$, then substituting~$\varphi(y)\1_{(x_0,+\infty)}(y)$ with~$\varphi(y+\eps)\1_{(x_0,+\infty)}(y)$ gives
\begin{multline*}
-\int_{-\infty}^{x_0}\int_{x_0-\eps}^{+\infty}\frac{\varphi(x)\,\varphi(y+\eps)}{{(y-x)}^{1+2s}}\;dy\;dx=
-\int_{-\infty}^{x_0}\int_{x_0}^{+\infty}\frac{\varphi(x)\,\varphi(y)}{{(y-x-\eps)}^{1+2s}}\;dy\;dx\leq\\
\leq
-\int_{-\infty}^{x_0}\int_{x_0}^{+\infty}\frac{\varphi(x)\,\varphi(y)}{{(y-x)}^{1+2s}}\;dy\;dx.
\end{multline*}
We then deduce that~$\lambda(\Omega)$ is minimized when~$\Omega$ is connected.

We now prove the second claim in the statement. In order to do so, we only need to prove that,
for any~$x_0\in(-1,1)$, it holds
\begin{align*}
\lambda(-1,1)<\lambda\big((-1,1)\setminus\{x_0\}\big).
\end{align*} 
This is a direct consequence of the strong maximum principle in~$(-1,1)$,
which can be stated as a corollary of the positivity of the Green function,
see~\cite{ajs-boggio}*{Theorem 1.1}.
\end{proof}

\subsection{When the integer part of \texorpdfstring{$s$}{s} is odd}

\begin{proposition}\label{prop:finitely many intervals}
Let~$\lfloor s\rfloor\in 2\N_0+1$. 
Suppose that~$\Omega\subset\R$ is an open set with~$|\Omega|=2$ and
\[
\inf\big\{\dist(I_k,I_h):\;k,h\in\cM,\;k\neq h\big\}>0.
\]
Then~$\lambda(-1,1)\leq\lambda(\Omega)$ and equality holds if and only if~$\Omega$ is an interval.
\end{proposition}
\begin{proof}
Suppose that~$\Omega$ is not connected. In this case,~$\varphi$ changes sign in~$\Omega$: this simply follows from Lemma~\ref{lem:one sign} and~\eqref{disjointsupport}.
Moreover, by assumption there exists~$\delta_0>0$ such that
\begin{align*}
\dist(I_k,I_h)\geq 2\delta_0
\qquad\text{for any }k,h\in\cM,\ k\neq h.
\end{align*}
By Lemma~\ref{lem:one sign}, we have that~$\varphi^+,\varphi^-\in\cH^s_0(\Omega)$.
For~$\delta\in(-\delta_0,\delta_0)$, consider the function
\begin{align*}
u(x)=\varphi^+(x)-\varphi^-(x-\delta),
\qquad x\in\R.
\end{align*}
It holds~$u\in H^s(\R)$ and, moreover,~$\|u\|_{L^2(\R)}=\|\varphi\|_{L^2(\R)}$.
Also,
\begin{align*}
\cE_s(u,u) &= \cE_s(\varphi^+,\varphi^+)+\cE_s(\varphi^-,\varphi^-)
-c_{1,s}\int_\R\int_\R\frac{\varphi^+(x)\,\varphi^-(y-\delta)}{{|x-y|}^{1+2s}}\;dx\;dy \\
&= \cE_s(\varphi^+,\varphi^+)+\cE_s(\varphi^-,\varphi^-)
-c_{1,s}\int_\R\int_\R\frac{\varphi^+(x)\,\varphi^-(y)}{{|x-y-\delta|}^{1+2s}}\;dx\;dy.
\end{align*}
We now differentiate twice in~$\delta$ and obtain
\begin{align*}
& \frac1{c_{1,s}}\frac{d^2}{d\delta^2}\bigg|_{\delta=0}\cE_s(u,u)=
-(1+2s)\frac{d}{d\delta}\bigg|_{\delta=0}
\int_\R\int_\R\frac{\varphi^+(x)\,\varphi^-(y)}{{|x-y-\delta|}^{2+2s}}\,\mathrm{sgn}(x-y-\delta)\;dx\;dy \\
& =-(1+2s)\frac{d}{d\delta}\bigg|_{\delta=0}
\int_\R\int_y^{+\infty}\frac{\varphi^+(x)\,\varphi^-(y)}{{|x-y-\delta|}^{2+2s}}\;dx\;dy \\
&\qquad +(1+2s)\frac{d}{d\delta}\bigg|_{\delta=0}
\int_\R\int_{-\infty}^y\frac{\varphi^+(x)\,\varphi^-(y)}{{|x-y-\delta|}^{2+2s}}\;dx\;dy \\
& = -(1+2s)(2+2s)\int_\R\int_y^{+\infty}\frac{\varphi^+(x)\,\varphi^-(y)}{{|x-y|}^{3+2s}}\;dx\;dy \\
&\qquad -(1+2s)(2+2s)\int_\R\int_{-\infty}^y\frac{\varphi^+(x)\,\varphi^-(y)}{{|x-y-\delta|}^{3+2s}}\;dx\;dy \\
& = -(1+2s)(2+2s)\int_\R\int_\R\frac{\varphi^+(x)\,\varphi^-(y)}{{|x-y|}^{3+2s}}\;dx\;dy 
\ < 0.
\end{align*}
From this we deduce that, at least for some~$\delta\in(-\delta_0,\delta_0)$,
it holds~$\cE_s(u,u)<\cE_s(\varphi,\varphi)$.
\end{proof}

\begin{theorem}\label{thm:faber-krahn odd}
Let~$\lfloor s\rfloor\in 2\N_0+1$. 
Then~$\lambda(-1,1)\leq\lambda(\Omega)$ for any open set~$\Omega\subset\R$ with~$|\Omega|=2$.
\end{theorem}
\begin{proof}

By Proposition~\ref{prop:finitely many intervals} and Lemma~\ref{lem:fill the holes},
we only need to prove that the claim holds when~$\cM$ is countable and 
\[
\inf\big\{\dist(I_k,I_h):\;k,h\in\cM,\;k\neq h\big\}=0.
\]
We proceed via a continuity argument. 


Write 
\begin{align*}
\Omega &= \bigcup_{k\in\N} I_k,
\qquad\text{where~$I_k$'s are disjoint open intervals,} &
\Omega_N &= \bigcup_{k=1}^N I_k,\qquad\text{for any }N\in\N, \\
\varphi &= \sum_{k\in\N} \varphi_k,
\qquad \varphi_k\in\cH^s_0(I_k), \text{ for any } k\in\N, &
u_N &=\sum_{k=1}^N \varphi_k,\qquad\text{for any }N\in\N.
\end{align*}
By the inclusion~$\Omega_N\subset\Omega$, one simply has~$\lambda(\Omega_N)\geq\lambda(\Omega)$.
Moreover,
\begin{align*}
\lambda(\Omega)
&=
\cE_s(\varphi,\varphi)=\cE_s(u_N,u_N)+2\cE_s(\varphi-u_N,u_N)+\cE_s(\varphi-u_N,\varphi-u_N) \\
&\geq 
\lambda(\Omega_N)\|u_N\|^2_{L^2(\R^n)}+2\cE_s(\varphi-u_N,u_N)+\cE_s(\varphi-u_N,\varphi-u_N).
\end{align*}
By Proposition~\ref{prop:finitely many intervals}, we have that~$\lambda(\Omega_N)$ is greater than the first eigenvalue of the ball with volume~$|\Omega_N|$, and then by the scaling properties of $\Ds$
\begin{align*}
\lambda(\Omega_N)\geq\bigg(\frac2{|\Omega_N|}\bigg)^{2s}\lambda(-1,1).
\end{align*}
As we have~$u_N\to\varphi$ in~$H^s(\R^n)$, indeed by~\eqref{eq:energies} and since~$\varphi-u_N$ is supported in~$\Omega\setminus\Omega_N$
\begin{align*}
 \cE_s(\varphi-u_N,\varphi-u_N)&=\cE_{s-\intgr s}\big(D(\varphi-u_N),D(\varphi-u_N)\big)
\qquad \text{where}\quad D=\nabla{(-\lapl)}^{(\intgr s-1)/2} \\
& = \frac{c_{1,s-\intgr s}}2
\int_{\R}\int_{\R}\frac{\big|D(\varphi-u_N)(x)-D(\varphi-u_N)(y)\big|^2}{{|x-y|}^{1+2s-2\intgr s}}\;dx\;dy \\
& =\frac{c_{1,s-\intgr s}}2
\int_{\Omega\setminus\Omega_N}\int_{\Omega\setminus\Omega_N}\frac{\big|D\varphi(x)-D\varphi(y)\big|^2}{{|x-y|}^{1+2s-2\intgr s}}\;dx\;dy \\
& \quad +c_{1,s-\intgr s}
\int_{\R\setminus(\Omega\setminus\Omega_N)}\int_{\Omega\setminus\Omega_N}\frac{\big|D\varphi(x)\big|^2}{{|x-y|}^{1+2s-2\intgr s}}\;dx\;dy \\
& \leq 
\frac{c_{1,s-\intgr s}}2
\int_{\Omega\setminus\Omega_N}\int_{\Omega\setminus\Omega_N}\frac{\big|D\varphi(x)-D\varphi(y)\big|^2}{{|x-y|}^{1+2s-2\intgr s}}\;dx\;dy \\
& \quad +c_{1,s-\intgr s}
\int_{\Omega\setminus\Omega_N}\big|D\varphi(x)\big|^2\,\mathrm{dist}\big(x,\R\setminus(\Omega\setminus\Omega_N)\big)^{-2s+2\intgr s}\;dx
\end{align*}
which converges to~$0$ as~$N\uparrow\infty$ and~$\Omega_N\nearrow\Omega$ by for example~\cite{triebel}*{Section 4.3.2, equation (7)},
we deduce
\begin{multline*}
\lambda(\Omega) =
\lim_{N\uparrow\infty}\Big[\lambda(\Omega_N)\|u_N\|^2_{L^2(\R^n)}+2\cE_s(\varphi-u_N,u_N)+\cE_s(\varphi-u_N,\varphi-u_N)\Big] \geq \\
\geq 
\lim_{N\uparrow\infty}\bigg(\frac2{|\Omega_N|}\bigg)^{2s}\|u_N\|^2_{L^2(\R^n)}\lambda(-1,1)=\lambda(-1,1).
\end{multline*}
\end{proof}

\appendix

\section{Projections onto the positive cone}
\label{sec:proj}

Denote by
\begin{align}\label{positive cone}
\cC^+(\Omega):=\left\lbrace v\in\cH^s_0(\Omega):v\geq 0 \right\rbrace
\end{align}
and note how any~$\varphi\in C^\infty_c(\Omega)$,~$\varphi\geq 0$, lies inside~$\cC^+(\Omega).$
Define the map
\begin{align*}
\pi^+:\cH^s_0(\Omega) \longrightarrow \cC^+(\Omega) 
\end{align*}
which associates to~$w$ its projection on~$\cC^+(\Omega)$,
that is 
\begin{align}\label{projection property}
\cE_s(w-\pi^+w,w-\pi^+w)\leq\cE_s(w-v,w-v),
\quad\text{for any }\ v\in\cC^+(\Omega).
\end{align}

We list here below a number of inequalities that are needed throughout the paper:
\begin{align}
\cE_s(w-\pi^+w,v-\pi^+w) &\leq 0 & \qquad \text{ for any }\ v\in\cC^+(\Omega) \label{eq:char} \\
0 &\leq \cE_s(\pi^+w-w,v)  & \qquad \text{ for any }\ v\in\cC^+(\Omega) \label{eq:ppp} \\
\cE_s(\pi^+w-w,\pi^+w) &= 0 \label{eq:test0} \\
\cE_s(w,\pi^+w) &\leq \cE_s(w,w) \label{eq:test1} \\
\cE_s(w,v) &\leq \cE_s(\pi^+w,v) & \qquad \text{ for any }\ v\in\cC^+(\Omega) \label{eq:upward} \\
\cE_s(\pi^+w,\pi^+w) &\leq \cE_s(w,w) \label{eq:downward1} \\
\cE_s(2\pi^+w-w,2\pi^+w-w) &\leq \cE_s(w,w) \label{eq:downward2}
\end{align}

\begin{proof}[Proof of~\eqref{eq:char}]
By convexity of~$\cC^+(\Omega)$
\begin{align*}
\pi^+w+t(v-\pi^+w)\in\cC^+(\Omega),
\quad\text{for any }\ t\in[0,1],\ v\in\cC^+(\Omega).
\end{align*}
So,
\begin{multline*}
\cE_s(w-\pi^+w,w-\pi^+w)\leq\cE_s(w-\pi^+w-t(v-\pi^+w),w-\pi^+w-t(v-\pi^+w))\ =\\
=\ \cE_s(w-\pi^+w,w-\pi^+w)-2t\cE_s(w-\pi^+w,v-\pi^+w)+t^2\cE_s(v-\pi^+w,v-\pi^+w),
\end{multline*}
which implies
\[
2\cE_s(w-\pi^+w,v-\pi^+w)\leq t\cE_s(v-\pi^+w,v-\pi^+w),
\quad\text{for any }\ t\in(0,1]
\]
and we deduce~\eqref{eq:char} by considering~$t$ arbitrarily small.
\end{proof}

\begin{proof}[Proof of~\eqref{eq:ppp}]
From~\eqref{eq:char} it follows
\begin{align*}
\cE_s(\pi^+w-w,v)\geq\cE_s(\pi^+w-w,\pi^+w)
\end{align*}
where the right-hand side is nonnegative again thanks to~\eqref{eq:char} by considering the particular case~$v=0$.
\end{proof}

\begin{proof}[Proof of~\eqref{eq:test0}]
As mentioned above, inequality~$\leq$ follows from~\eqref{eq:char} by considering the particular case~$v=0$.
Inequality~$\geq$ follows from~\eqref{eq:ppp} by considering the particular case~$v=\pi^+w$.
\end{proof}

\begin{proof}[Proof of~\eqref{eq:test1}]
Using again~\eqref{eq:test0} we deduce
\begin{align*}
\cE_s(w,\pi^+w-w)=\cE_s(w-\pi^+w,\pi^+w-w)+\cE_s(\pi^+w,\pi^+w-w) \leq 0.
\end{align*}
\end{proof}

\begin{proof}[Proof of~\eqref{eq:upward}]
Since both~$v$ and~$\pi^+w$ belong to~$\cC^+(\Omega)$, then so does~$2\pi^+w+v$. Then, by~\eqref{eq:char},
\begin{multline*}
0\geq\cE_s(w-\pi^+w,(2\pi^+w+v)-\pi^+w)=\cE_s(w-\pi^+w,\pi^+w+v)=\\
=\cE_s(w-\pi^+w,\pi^+w)+\cE_s(w-\pi^+w,v)
\end{multline*}
and since, by~\eqref{eq:test0}, it holds~$\cE_s(w-\pi^+w,\pi^+w)\geq 0$,
we deduce~$\cE_s(w-\pi^+w,v)\leq 0$.
\end{proof}

\begin{proof}[Proof of~\eqref{eq:downward1}]
Write (exploiting~\eqref{eq:test0} in all inequalities below)
\begin{multline*}
\cE_s(\pi^+w,\pi^+w)-\cE_s(w,w)=\cE_s(\pi^+w-w,\pi^+w)+\cE_s(w,\pi^+w)-\cE_s(w,w)\leq\ \\
\leq\ \cE_s(w,\pi^+w-w) = \cE_s(w-\pi^+w,\pi^+w-w)+\cE_s(\pi^+w,\pi^+w-w)\leq 0.
\end{multline*}
\end{proof}

\begin{proof}[Proof of~\eqref{eq:downward2}]
Write
\begin{align*}
\cE_s(2\pi^+w-w,2\pi^+w-w)=\cE_s(w,w)+4\cE_s(\pi^+w-w,\pi^+w)\leq \cE_s(w,w)
\end{align*}
where we have used~\eqref{eq:test0}.
\end{proof}

\end{document}